\documentclass[12pt, twoside, leqno]{article}
\usepackage{amsmath,amsthm}
\usepackage{amssymb}
\usepackage{enumitem}
\usepackage{graphicx}
\usepackage[T1]{fontenc}

\newtheorem{theorem}{Theorem}[section]
\newtheorem{corollary}[theorem]{Corollary}
\newtheorem{lemma}[theorem]{Lemma}

\theoremstyle{definition}
\newtheorem{definition}[theorem]{Definition}
\newtheorem{remark}[theorem]{Remark}

\numberwithin{equation}{section}
\begin{document}

\title{On the composition operators on Besov and Triebel-Lizorkin spaces of
power weights}
\author{Douadi Drihem \\
M'sila University, Department of Mathematics\\
Laboratory of Functional Analysis and Geometry of Spaces\\
\ M'sila 28000, Algeria\\
E-mail: douadidr@yahoo.fr, douadi.drihem@univ-msila.dz}
\date{}
\maketitle

\begin{abstract}
Let $G:\mathbb{R\rightarrow R}$ be a continuous function. Under some
assumptions on $G$, $s,\alpha ,p$ and $q$ we prove that 
\begin{equation*}
\{G(f):f\in A_{p,q}^{s}(\mathbb{R}^{n},|\cdot |^{\alpha })\}\subset
A_{p,q}^{s}(\mathbb{R}^{n},|\cdot |^{\alpha })
\end{equation*}%
implies that $G$ is a linear function. Here $A_{p,q}^{s}(\mathbb{R}%
^{n},|\cdot |^{\alpha })$ stands  either for the  Besov space $B_{p,q}^{s}(%
\mathbb{R}^{n},|\cdot |^{\alpha })$ or for the Triebel-Lizorkin space $%
F_{p,q}^{s}(\mathbb{\ R}^{n},|\cdot |^{\alpha })$. These spaces unify and
generalize many classical function spaces such as Sobolev spaces of power
weights.
\end{abstract}


\baselineskip=17pt



\renewcommand{\thefootnote}{}

\footnote{%
}


\renewcommand{\thefootnote}{\arabic{footnote}} \setcounter{footnote}{0}


\section{Introduction}

Let $G:\mathbb{R\rightarrow R}$ be a $C^{\infty }$\ function. Suppose that $%
1\leq p<\frac{n}{m}\ $for $m\geq 3$ and $1<p<\frac{n}{2}$ for $m=2$. In 1978
Dahlberg \cite[Theorem 1]{Da79} proved that%
\begin{equation*}
G(f)\in W_{p}^{m}(\mathbb{R}^{n}),\quad f\in W_{p}^{m}(\mathbb{R}^{n}),
\end{equation*}%
implies $G(t)=ct$ for some $c\in \mathbb{R}$. More precisely, there is no
non-trivial function $G$ which acts via left composition on $W_{p}^{m}(%
\mathbb{R}^{n})$ spaces, with $1\leq p<\frac{n}{m}\ $for $m\geq 3$ and $1<p<%
\frac{n}{2}$ for $m=2$.

The extension of the Dahlberg result to Besov and Triebel-Lizorkin spaces is
given by Bourdaud in \cite{Bo93} and \cite{Bo931}, Runst in \cite{Ru86}, and
Sickel in \cite{Si97}, \cite{Si98} and \cite{Si98-1}. For a continuous
function $G:\mathbb{R}\rightarrow \mathbb{R}$ and Lebesgue-measurable $f$ we
shall call the operator%
\begin{equation*}
T_{G}:f\rightarrow G(f),
\end{equation*}%
the composition operator or the Nemytzkij operator. \ Further results
concerning the composition operators in Besov and Triebel-Lizorkin spaces
are given \cite{BK}, \cite{BCS06}, \cite{BMS10} and \cite{RS96}.

Recently the author in \cite{Dr21} gave  necessary and sufficient
conditions on $G$ such that 
\begin{equation}
T_{G}(W_{p}^{m}(\mathbb{R}^{n},|\cdot |^{\alpha }))\subset W_{p}^{m}(\mathbb{%
\ R}^{n},|\cdot |^{\alpha }),  \label{Dah1}
\end{equation}
with some suitable assumptions on $m,p$ and $\alpha $. More precisely, he
proved the following result. Let $m=2,3,...$ and let $1<p<\infty ,0\leq
\alpha <n(p-1)$. Assume that $m>\frac{n+\alpha }{p}$. Then the composition
operator $T_{G}$ satisfies $\mathrm{\eqref{Dah1}}$ if and only if $G$
satisfies the following conditions: 
\begin{equation*}
G(0)=0\quad \text{and}\quad G^{(m)}\in \mathbb{L}_{\mathrm{loc}}^{p}(\mathbb{%
\ R}).
\end{equation*}%
%
%
%
%
%
%
%
%
%
%
%
For the classical Sobolev spaces, see \cite{Bo91}, \cite{Bo10}, \cite%
{Igari65} and \cite{MM}.

The motivation to study the problem of composition on function spaces comes
from applications to partial differential equations, where many nonlinear
equations are given by a composition operator, for example the nonlinear
equations 
\begin{equation*}
\left\{ 
\begin{array}{ccc}
\partial _{t}f(t,x)-\Delta f(t,x)= & T_{G}(f(t,x)), & \quad (t,x)\in \mathbb{%
\ R}^{+}\mathbb{\times R}^{n}. \\ 
f(0,x)=f_{0}(x). &  & 
\end{array}%
\right.
\end{equation*}%
To study this equation in functional spaces such as Sobolev spaces we need
to estimate the nonlinear term $T_{G}$ in such spaces, see for example \cite%
{F98}.

Another motivation to study the composition operators in function spaces can
be found in \cite{CFZ11} and the references therein.

This paper is a continuation of the previous paper \cite{Dr21} written by
the same author. We will study the Dahlberg problem on Besov and
Triebel-Lizorkin spaces with power weights. One of the main difficulties to
study this problem is that the norm of the $A_{p,q}^{s}(\mathbb{R}%
^{n},|\cdot |^{\alpha })$ spaces with $\alpha \neq 0$ is not translation
invariant, so some new techniques must be developed. Our main theorem of
this paper is the following.

\begin{theorem}
\label{Triviality1}Let\ $1<p<\infty ,0<q\leq \infty \ $and $0\leq \alpha
<n(p-1)$. Let $G\in C^{2}(\mathbb{R})$.\ Suppose 
\begin{equation*}
1+\frac{1}{p}<s<\frac{n+\alpha }{p}
\end{equation*}%
and 
\begin{equation}
T_{G}(\mathbb{A}_{p,q}^{s}(\mathbb{R}^{n},|\cdot |^{\alpha }))\subset 
\mathbb{A}_{p,q}^{s}(\mathbb{R}^{n},|\cdot |^{\alpha }).  \label{Condition2}
\end{equation}%
Then 
\begin{equation}
G(t)=ct,\quad t\in \mathbb{R}  \label{linear}
\end{equation}%
for some constant $c$.
\end{theorem}

Here $\mathbb{A}_{p,q}^{s}(\mathbb{R}^{n},|\cdot |^{\alpha })$ stands either
for the Besov space $\mathbb{B}_{p,q}^{s}(\mathbb{R}^{n},|\cdot |^{\alpha })$ or
for the Triebel-Lizorkin space $\mathbb{F}_{p,q}^{s}(\mathbb{R}^{n},|\cdot
|^{\alpha })$. We recover the results on classical Besov and
Triebel-Lizorkin spaces by taking $\alpha =0$. Some sufficient conditions on 
$G$ which ensure $\mathrm{\eqref{Condition2}}$ are given in \cite{DrBanach}.

The question arises what happens when $s\geq \frac{n+\alpha }{p}$ holds. In
that case the Dahlberg result does not hold by the following theorem.

\begin{theorem}
\label{Triviality1 copy(1)}Let\ $1< p<\infty ,1\leq q\leq \infty \ $and $%
0\leq \alpha <n(p-1)$. Let $G(t)=t^{2},t\in \mathbb{R}$.\ Suppose that $s> 
\frac{n+\alpha }{p}$ or 
\begin{equation*}
s=\frac{n+\alpha }{p}\quad \text{and}\quad q=1
\end{equation*}
in the case of Besov spaces $B_{p,q}^{s}(\mathbb{R}^{n},|\cdot |^{\alpha })$
. Then 
\begin{equation*}
T_{G}(\mathbb{A}_{p,q}^{s}(\mathbb{R}^{n},|\cdot |^{\alpha }))\subset 
\mathbb{A}_{p,q}^{s}(\mathbb{R}^{n},|\cdot |^{\alpha }).
\end{equation*}
\end{theorem}

We will prove these results in Section 3.

\subsection{Notation}

As usual, we denote by $\mathbb{R}^{n}$ the $n$-dimensional real Euclidean
space, $\mathbb{N}$ the collection of all natural numbers and $\mathbb{N}%
_{0}=\mathbb{N}\cup \{0\}$. For a multi-index $\alpha =\left( \alpha
_{1},...,\alpha _{n}\right) \in \mathbb{N}_{0}^{n}$, we write $\left\vert
\alpha \right\vert =\alpha _{1}+...+\alpha _{n}$. The notation $%
X\hookrightarrow Y$ stands for continuous embeddings from $X$ to $Y$, where $%
X$ and $Y$ are normed spaces. We use the notation $\lfloor x\rfloor $ for
the integer part of the real number $x$. Let $f$ be a measurable function
and $a\in \mathbb{R}^{n}$. We define the translation operator by $\tau
_{a}f=f(\cdot -a)$.

For $x\in \mathbb{R}^{n}$ and $r>0$ we denote by $B(x,r)$ the open ball in $%
\mathbb{R}^{n}$ with center $x$ and radius $r$.

If $\Omega \subset {\mathbb{R}^{n}}$ is a measurable set, then $|\Omega |$
stands for the (Lebesgue) measure of $\Omega $ and $\chi _{\Omega }$ denotes
its characteristic function. The Lebesgue space $L^{p}$, $0<p\leq \infty $
consists of all measurable functions $f$ for which 
\begin{equation*}
\big\|f\big\|_{p}=\Big(\int_{\mathbb{R}^{n}}\left\vert f(x)\right\vert ^{p}dx%
\Big)^{1/p}<\infty ,\text{\quad }0<p<\infty
\end{equation*}%
and 
\begin{equation*}
\big\|f\big\|_{\infty }=\underset{x\in \mathbb{R}^{n}}{\text{ess-sup}}%
\left\vert f(x)\right\vert <\infty .
\end{equation*}

Let $\alpha \in \mathbb{R}$ and $0<p<\infty $. The weighted Lebesgue space $%
L^{p}(\mathbb{R}^{n},|\cdot |^{\alpha })$ contains all measurable functions $%
f$ such that 
\begin{equation*}
\big\|f\big\|_{L^{p}(\mathbb{R}^{n},|\cdot |^{\alpha })}=\Big(\int_{\mathbb{R%
}^{n}}\left\vert f(x)\right\vert ^{p}|x|^{\alpha }dx\Big)^{1/p}<\infty .
\end{equation*}

By $\mathcal{S}(\mathbb{R}^{n})$ we denote the Schwartz space of all
complex-valued, infinitely differentiable and rapidly decreasing functions
on $\mathbb{R}^{n}$ and by $\mathcal{S}^{\prime }(\mathbb{R}^{n})$ the dual
space of all tempered distributions on $\mathbb{R}^{n}$. We define the
Fourier transform of a function $f\in \mathcal{S}(\mathbb{R}^{n})$ by%
\begin{equation*}
\mathcal{F}(f)(\xi )=\left( 2\pi \right) ^{-n/2}\int_{\mathbb{R}
^{n}}e^{-ix\cdot \xi }f(x)dx,\quad \xi \in \mathbb{R}^{n}.
\end{equation*}%
The inverse Fourier transform is denoted by $\mathcal{F}^{-1}f$. Both $%
\mathcal{F}$ and $\mathcal{F}^{-1}$ are extended to the dual Schwartz space $%
\mathcal{S}^{\prime }(\mathbb{R}^{n})$ in the usual way.

By $c$ we denote generic positive constants, which may have different values
at different occurrences. Further notation will be introduced later on when
needed.

\section{Besov and Triebel-Lizorkin spaces}

We present the Fourier analytical definition of Besov and Triebel-Lizorkin
spaces of power weights and recall their basic properties. We first need the
concept of a smooth dyadic resolution of unity. Let $\psi $\ be a function\
in $\mathcal{S}(\mathbb{R}^{n})$\ satisfying%
\begin{equation*}
0\leq \psi \leq 1\text{\quad and}\quad \psi (x)=\left\{ 
\begin{array}{ccc}
1, & \text{if } & \left\vert x\right\vert \leq 1, \\ 
0, & \text{if} & \left\vert x\right\vert \geq \frac{3}{2}.%
\end{array}%
\right. 
\end{equation*}%
We put $\mathcal{F}\varphi _{0}=\psi $, $\mathcal{F}\varphi _{1}=\psi (\frac{%
\cdot }{2})-\psi $\ and $\mathcal{F}\varphi _{j}=\mathcal{F}\varphi
_{1}(2^{1-j}\cdot )\ $for$\ j=2,3$,.... Then $\{\mathcal{F}\varphi
_{j}\}_{j\in \mathbb{N}_{0}}$\ is a smooth dyadic resolution of unity, 
\begin{equation*}
\sum_{j=0}^{\infty }\mathcal{F}\varphi _{j}(x)=1
\end{equation*}%
for all $x\in \mathbb{R}^{n}$.\ Thus we obtain the Littlewood-Paley
decomposition 
\begin{equation*}
f=\sum_{j=0}^{\infty }\varphi _{j}\ast f
\end{equation*}%
of all $f\in \mathcal{S}^{\prime }(\mathbb{R}^{n})$ $($convergence in $%
\mathcal{S}^{\prime }(\mathbb{R}^{n}))$.

We are now in a position to state the definition of Besov and
Triebel-Lizorkin spaces equipped with power weights.

\begin{definition}
\label{def-herz-Besov}Let $\alpha ,s\in \mathbb{R}$, $0<p<\infty $\ and $%
0<q\leq \infty $. $\newline
\mathrm{(i)}$ The Besov space $B_{p,q}^{s}(\mathbb{R}^{n},|\cdot |^{\alpha
}) $ is the collection of all $f\in \mathcal{S}^{\prime }(\mathbb{R}^{n})$
such that 
\begin{equation*}
\big\Vert f\big\Vert_{B_{p,q}^{s}(\mathbb{R}^{n},|\cdot |^{\alpha })}=\Big(%
\sum_{j=0}^{\infty }2^{jsq}\big\Vert\varphi _{j}\ast f\big\Vert_{L^{p}(%
\mathbb{R}^{n},|\cdot |^{\alpha })}^{q}\Big)^{1/q}<\infty ,
\end{equation*}%
with the obvious modification if $q=\infty $.$\newline
\mathrm{(ii)}$ The Triebel-Lizorkin space $F_{p,q}^{s}(\mathbb{R}^{n},|\cdot
|^{\alpha })$ is the collection of all $f\in \mathcal{S}^{\prime }(\mathbb{R}%
^{n})$\ such that 
\begin{equation*}
\big\Vert f\big\Vert_{F_{p,q}^{s}(\mathbb{R}^{n},|\cdot |^{\alpha })}=\Big\|%
\Big(\sum\limits_{j=0}^{\infty }2^{jsq}\left\vert \varphi _{j}\ast
f\right\vert ^{q}\Big)^{1/q}\Big\|_{L^{p}(\mathbb{R}^{n},|\cdot |^{\alpha
})}<\infty ,
\end{equation*}%
with the obvious modification if $q=\infty .$
\end{definition}

\begin{remark}
Let\textrm{\ }$s\in \mathbb{R},0<p<\infty ,0<q\leq \infty $ and $\alpha >-n$%
. The spaces\textrm{\ }$B_{p,q}^{s}(\mathbb{R}^{n},|\cdot |^{\alpha })$ and $%
F_{p,q}^{s}(\mathbb{R}^{n},|\cdot |^{\alpha })$ are independent of the
particular choice of the smooth dyadic resolution of unity \textrm{\ }$\{%
\mathcal{F}\varphi _{j}\}_{j\in \mathbb{N}_{0}}$ (in the sense of$\mathrm{\ }
$equivalent quasi-norms). In particular $B_{p,q}^{s}(\mathbb{R}^{n},|\cdot
|^{\alpha })$ and $F_{p,q}^{s}(\mathbb{R}^{n},|\cdot |^{\alpha })$ are
quasi-Banach spaces and they are Banach spaces if $p,q\geq 1$, see \cite%
{Bui82} and \cite{XuYang03}. In addition 
\begin{equation}
F_{p,2}^{m}(\mathbb{R}^{n},|\cdot |^{\alpha })=W_{p}^{m}(\mathbb{R}%
^{n},|\cdot |^{\alpha }),\quad (\text{Sobolev spaces of power weights})
\label{Sobolev}
\end{equation}%
for any $m\in \mathbb{N}_{0},1<p<\infty $ and any $-n<\alpha <n(p-1)$.
Moreover, for $\alpha =0$ we re-obtain the usual Besov and Triebel-Lizorkin
spaces, see \cite{Sawano18}, \cite{Triebel83} and \cite{Triebel92} for more
details about unweighted function spaces.
\end{remark}

The next theorem implies that the spaces $A_{p,q}^{s}(\mathbb{R}^{n},|\cdot
|^{\alpha })$ exclusively contain regular distributions, at least for $1\leq
p<\infty ,1\leq q\leq \infty ,\alpha \geq 0$ and $s>0$. The proof is given
in \cite{Drihem20}.

\begin{theorem}
\label{embeddings}\textit{Let }$1< p<\infty ,1\leq q\leq \infty ,0\leq
\alpha <n(p-1)$ and $s>0$. Then 
\begin{equation*}
A_{p,q}^{s}(\mathbb{R}^{n},|\cdot |^{\alpha })\subset L_{\mathrm{loc} }^{1}(%
\mathbb{R}^{n}).
\end{equation*}
\end{theorem}

Let $f$ be an arbitrary function on $\mathbb{R}^{n}$ and $x,h\in \mathbb{R}%
^{n}$. We put 
\begin{equation*}
\Delta _{h}f(x)=f(x+h)-f(x).
\end{equation*}
By ball means of differences we mean the function%
\begin{equation*}
d_{t}^{M}f(x)=t^{-n}\int_{|h|\leq t}\left\vert \Delta
_{h}^{M}f(x)\right\vert dh=\int_{B}\left\vert \Delta
_{th}^{M}f(x)\right\vert dh,
\end{equation*}%
where $B=\{y\in \mathbb{R}^{n}:|h|\leq 1\}$ is the unit ball of $\mathbb{R}%
^{n}$, $t>0$ is a real number and $M$ is a natural number. Let $1\leq
p<\infty ,1\leq q\leq \infty ,\alpha \geq 0$ and $s>0$. We set%
\begin{equation*}
\big\|f\big\|_{B_{p,q}^{s}(\mathbb{R}^{n},|\cdot |^{\alpha })}^{\ast }=\big\|%
f\big\|_{L^{p}(\mathbb{R}^{n},|\cdot |^{\alpha })}+\Big(\int_{0}^{1}t^{-sq}%
\big\|d_{t}^{M}f\big\|_{L^{p}(\mathbb{R}^{n},|\cdot |^{\alpha })}^{q}\frac{dt%
}{t}\Big)^{1/q}
\end{equation*}%
with the obvious modification if $q=\infty .$

To prove Theorem \ref{Triviality1} with $\alpha =0$, the arguments of \cite[%
Corollary 2]{Si97} are based on the characterization of Besov and
Triebel-Lizorkin spaces by differences. One of the main diffculties to prove
Theorem \ref{Triviality1} is that the norm in $A_{p,q}^{s}(\mathbb{R}%
^{n},|\cdot |^{\alpha })$ with $\alpha \neq 0$ is not translation invariant,
so we are forced to introduce a new method. We think that it is better to
use the following characterization of $B_{p,q}^{s}(\mathbb{R}^{n},|\cdot
|^{\alpha })$ by ball means of differences, see \cite{DjDr}.

\begin{theorem}
\label{means-diff-cha1}\textit{Let }$1< p<\infty ,1\leq q\leq \infty
,0<\alpha <n(p-1) $ and $M\in \mathbb{N}.$ Assume that 
\begin{equation*}
0<s<M.
\end{equation*}
Then $\big\|\cdot \big\|_{B_{p,q}^{s}(\mathbb{R}^{n},|\cdot |^{\alpha
})}^{\ast }$ is equivalent norm in $B_{p,q}^{s}(\mathbb{R}^{n},|\cdot
|^{\alpha })$.
\end{theorem}

To prove our results of this paper we need some embeddings. The following
statement holds by \cite[Theorem 5.9]{Drihem2013a} and \cite{MM12}.

\begin{theorem}
\label{embeddings5}\textit{Let }$\alpha _{1},\alpha _{2},s_{1},s_{2}\in 
\mathbb{R},1\leq \beta \leq \infty ,\alpha _{1}>-n\ $\textit{and }$\alpha
_{2}>-n$. \textit{We suppose that } 
\begin{equation*}
s_{1}-\frac{n+\alpha _{1}}{p}\leq s_{2}-\frac{n+\alpha _{2}}{q}.
\end{equation*}
\textit{Let }$1\leq q\leq p<\infty $ and $\frac{\alpha _{2}}{q}\geq \frac{%
\alpha _{1}}{p}$. Then 
\begin{equation*}
B_{q,\beta }^{s_{2}}(\mathbb{R}^{n},|\cdot |^{\alpha _{2}})\hookrightarrow
B_{p,\beta }^{s_{1}}(\mathbb{R}^{n},|\cdot |^{\alpha _{1}}).
\end{equation*}
\end{theorem}

\subsection{Technical results}

In this subsection we give several results used throughout this paper. The
Hardy-Littlewood maximal operator $\mathcal{M}$ is defined on locally
integrable functions by 
\begin{equation*}
\mathcal{M}f(x)=\sup_{r>0}\frac{1}{|B(x,r)|}\int_{B(x,r)}|f(y)|dy.
\end{equation*}%
Various important results have been proved in the space $L^{p}(\mathbb{R}%
^{n},|\cdot |^{\alpha })$ under some assumptions on $\alpha $ and $p$. The
condition $-n<\alpha <n(p-1),1<p<\infty $ is crucial in the study of the
boundedness of classical operators in $L^{p}(\mathbb{R}^{n},|\cdot |^{\alpha
})$ spaces, such as the Hardy-Littlewood maximal operator. One of the main
tools of this paper is based on the following results, which follows since $%
|\cdot |^{\alpha }\in \mathcal{A}_{p}(\mathbb{R}^{n})$, the Muckenhoupt
class, if and only if $-n<\alpha <n(p-1)$.

\begin{lemma}
\label{Fefferman-Stein}Let $1<p<\infty $ and $-n<\alpha <n(p-1)$. Then 
\begin{equation*}
\big\|\mathcal{M}f\big\|_{L^{p}(\mathbb{R}^{n},|\cdot |^{\alpha })}\lesssim %
\big\|f\big\|_{L^{p}(\mathbb{R}^{n},|\cdot |^{\alpha })}
\end{equation*}
holds\ for any $f\in L^{p}(\mathbb{R}^{n},|\cdot |^{\alpha })$.
\end{lemma}

For the proof, see e.g. \cite[p 218, 6.4]{St93}. We shall require below the
following lemma which is a simple conclusion of\ \cite{AJ81} and \cite{Ko}.

\begin{lemma}
\label{Fefferman-Stein copy(2)}Let $1<q<\infty $ and $1<p<\infty $. If $%
\{f_{j}\}_{j=0}^{\infty }$ is a sequence of locally integrable functions on $%
\mathbb{R}^{n}$ and $-n<\alpha <n(p-1)$, then
\end{lemma}

\begin{equation*}
\Big\|\Big(\sum_{j=0}^{\infty }(\mathcal{M}f_{j})^{q}\Big)^{1/q}%
\Big\|_{L^{p}(\mathbb{R}^{n},|\cdot |^{\alpha })}\lesssim \Big\|\Big(%
\sum_{j=0}^{\infty }|f_{j}|^{q}\Big)^{1/q}\Big\|_{L^{p}(\mathbb{R}
^{n},|\cdot |^{\alpha })}.
\end{equation*}
We need the following Marschall's inequalities, see {\cite[Proposition 1.3]%
{Mar91Forum}. But here, we use the simplified version given in \cite[
Proposition 6.1]{SiYY}.}

\begin{lemma}
\label{Marschall-ineq}Let\textit{\ }$A>0,R\geq 1$. Let $b\in \mathcal{D}( 
\mathbb{R}
^{n})$ and a function $f\in C^{\infty }(\mathbb{\ 
\mathbb{R}
}^{n})$ such that 
\begin{equation*}
\mathrm{supp}\mathcal{F}f\subseteq \left\{ \xi \in \mathbb{R}^{n}:\left\vert
\xi \right\vert \leq AR\right\} \quad \text{and}\quad \mathrm{supp}
b\subseteq \left\{ \xi \in \mathbb{R}^{n}:\left\vert \xi \right\vert \leq
A\right\} .
\end{equation*}
Then 
\begin{equation*}
\left\vert \mathcal{F}^{-1}b\ast f(x)\right\vert \leq c(AR)^{\frac{n}{t}-n} %
\big\|b\big\|_{\dot{B}_{1,t}^{\frac{n}{t}}}\mathcal{M}_{t}(f)(x)
\end{equation*}
for any $0<t\leq 1$ and any $x\in \mathbb{R}^{n}$, where $c$ is independent
of $A$, $R$, $b$ and $f$, and 
\begin{equation*}
\mathcal{M}_{t}(f)=\big(\mathcal{M}(|f|^{t})\big)^{1/t}.
\end{equation*}
\end{lemma}

Here $\dot{B}_{1,t}^{\frac{n}{t}}$ denotes the homogeneous Besov spaces.

The purpose of the following lemma is to generalize the dilation properties
obtained on Besov and Triebel-Lizorkin spaces in \cite[ Proposition 3.4.1]%
{Triebel83} to the spaces $A_{p,q}^{s}(\mathbb{R}^{n},|x|^{\alpha })$, which
will play an important role in the rest of the paper.

\begin{theorem}
\label{dilation}Let $1<p<\infty ,1<q\leq \infty ,-n<\alpha <n(p-1)$ and\ $%
s>0 $. Then there exists a positive constant $c$ independent of $\lambda $
such that 
\begin{equation*}
\big\|f(\lambda \cdot )\big\|_{A_{p,q}^{s}(\mathbb{R}^{n},|x|^{\alpha
})}\leq c\text{ }\lambda ^{s-\frac{n+\alpha }{p}}\big\|f\big\|_{A_{p,q}^{s}( 
\mathbb{R}^{n},|x|^{\alpha })}
\end{equation*}
holds for all $\lambda $ with $1\leq \lambda <\infty $ and all $f\in
A_{p,q}^{s}(\mathbb{R}^{n},|x|^{\alpha })$.
\end{theorem}

\begin{proof}
As the proof for the space $ B_{p,q}^{s}(\mathbb{R}^{n},|x|^{\alpha })$ is similar, we consider only $ F_{p,q}^{s}(\mathbb{R}^{n},|x|^{\alpha })$. Of course, $f(\lambda \cdot )$ must be interpreted in the sense of distributions. By Theorem \ref{embeddings}, it
follows that $f(x)$ is a regular distribution and $f(\lambda x)$ makes also
sense as a locally integrable function. Our proofs use partially some
decomposition techniques already used in \cite[ Proposition 3.4.1]{Triebel83}%
. Let $k\in \mathbb{N}_{0}$ be such that $2^{k}\leq \lambda <2^{k+1}$. Let $%
\{\mathcal{F}\varphi _{v}\}_{v\in \mathbb{N}_{0}}$\ be a smooth dyadic
resolution of unity. We have%
\begin{align*}
	\mathcal{F}^{-1}(\mathcal{F}\varphi _{v}\mathcal{F}(f(\lambda \cdot )))(x)
	&=\lambda ^{-n}\mathcal{F}^{-1}(\mathcal{F}\varphi _{v}\mathcal{F}%
	(f)(\lambda ^{-1}\cdot ))(x) \\
	&=\lambda ^{-n}\mathcal{F}^{-1}(\mathcal{F}\varphi _{1}(2^{1-v}\cdot )%
	\mathcal{F}(f)(\lambda ^{-1}\cdot ))(x) \\
	&=\mathcal{F}^{-1}(\mathcal{F}\varphi _{1}(\lambda 2^{1-v}\cdot )\mathcal{F%
	}(f))(\lambda x),\quad x\in \mathbb{R}^{n},
\end{align*}%
if $v\in \mathbb{N}$. Similarly if $v=0$. We will prove that%
\begin{align}
&\Big\|\Big(\sum\limits_{v=0}^{\infty }2^{vsq}\left\vert \mathcal{F}^{-1}(%
\mathcal{F}\phi _{v,\lambda }\mathcal{F}(f))(\lambda \cdot)\right\vert ^{q}\Big)%
^{1/q}\Big\|_{L^{p}(\mathbb{R}^{n},|x|^{\alpha })}  \label{Est-dila} \\
&\leq c\lambda ^{s-\frac{n+\alpha }{p}}\big\|f\big\|_{F_{p,q}^{s}(\mathbb{R}%
	^{n},|x|^{\alpha })},  \notag
\end{align}%
where the positive constant $c$ is independent of $\lambda $ and%
\begin{equation*}
\mathcal{F}\phi _{v,\lambda }=\left\{ 
\begin{array}{ccc}
\mathcal{F}\varphi _{0}(\lambda \cdot ), & \text{if} & v=0 \\ 
\mathcal{F}\varphi _{1}(\lambda 2^{1-v}\cdot ), & \text{if} & v\in \mathbb{N%
}.%
\end{array}%
\right. 
\end{equation*}%
We divide the sum in {$\mathrm{\eqref{Est-dila}}$ into two parts; the first, 
}$\sum_{v=0}^{k+2}$ and the second $\sum_{v=k+3}^{\infty }$. We have%
\begin{align*}
	\mathcal{F}\varphi _{1}(\lambda 2^{1-v}\cdot ) &=\sum_{j=v-2}^{v+1}%
	\mathcal{F}\varphi _{1}(\lambda 2^{1-v}\cdot )\mathcal{F}\varphi
	_{1}(2^{k-j+1}\cdot ) \\
	&=\sum_{i=-2}^{1}\mathcal{F}\phi _{v,\lambda }\mathcal{F}\varphi
	_{1}(2^{k-v-i+1}\cdot ), \quad v\geq k+3.
\end{align*}%

Thus,%
\begin{equation*}
\mathcal{F}^{-1}(\mathcal{F}\phi _{v,\lambda }\mathcal{F}(f))=%
\sum_{i=-2}^{1}\phi _{v,\lambda }\ast \varphi _{v-k+i}\ast f,
\quad v\geq k+3.
\end{equation*}%
By Lemma \ref{Marschall-ineq} we obtain%
\begin{equation*}
|\mathcal{F}^{-1}(\mathcal{F}\phi _{v,\lambda }\mathcal{F}(f))|\leq
\sum_{i=-2}^{1}\mathcal{M}(\varphi _{v-k+i}\ast f),
\quad v\geq k+3.
\end{equation*}%
Using Lemma \ref{Fefferman-Stein copy(2)}, we find
\begin{align*}
	&\Big\|\Big(\sum_{v=k+3}^{\infty }2^{vsq}\left\vert \mathcal{F}^{-1}(%
	\mathcal{F}\phi _{v,\lambda }\mathcal{F}(f))(\lambda \cdot)\right\vert ^{q}\Big)%
	^{1/q}\Big\|_{L^{p}(\mathbb{R}^{n},|x|^{\alpha })} \\
	&\leq \lambda ^{-\frac{n+\alpha }{p}}\sum_{i=-2}^{1}\Big\|\Big(%
	\sum_{v=k+3}^{\infty }2^{vsq}\left\vert \mathcal{M}(\varphi _{v-k+i}\ast
	f)\right\vert ^{q}\Big)^{1/q}\Big\|_{L^{p}(\mathbb{R}^{n},|x|^{\alpha })} \\
	&\lesssim \lambda ^{s-\frac{n+\alpha }{p}}\big\|f\big\|_{F_{p,q}^{s}(%
		\mathbb{R}^{n},|x|^{\alpha })},
\end{align*}%
where the implicit constant is independent of $\lambda $. Let $v\in
\{1,...,k-2\}$. Then

\begin{align*}
	\mathcal{F}^{-1}(\mathcal{F}\varphi _{v}\mathcal{F}(f(\lambda \cdot )))(x)
	&=\mathcal{F}^{-1}(\mathcal{F}\varphi _{1}(\lambda 2^{1-v}\cdot )\mathcal{F%
	}(f))(\lambda x) \\
	&=\mathcal{F}^{-1}(\mathcal{F}\varphi _{1}(\lambda 2^{1-v}\cdot )\mathcal{F%
	}\varphi _{0}\mathcal{F}(f))(\lambda x) \\
	&=\phi _{v,\lambda }\ast \varphi _{0}\ast f(\lambda x),\quad x\in \mathbb{R}%
	^{n}.
\end{align*}%
Hence, by Lemmas \ref{Fefferman-Stein} and \ref{Marschall-ineq} we obtain%
\begin{align*}
	\big\|\mathcal{F}^{-1}(\mathcal{F}\varphi _{v}\mathcal{F}(f(\lambda \cdot )))%
	\big\|_{L^{p}(\mathbb{R}^{n},|\cdot |^{\alpha })} &=\big\|\phi _{v,\lambda
	}\ast \varphi _{0}\ast f(\lambda \cdot )\big\|_{L^{p}(\mathbb{R}^{n},|\cdot
		|^{\alpha })} \\
	&= \lambda ^{-\frac{n+\alpha }{p}}\big\|\phi _{v,\lambda }\ast \varphi
	_{0}\ast f\big\|_{L^{p}(\mathbb{R}^{n},|\cdot |^{\alpha })} \\
	&\lesssim \lambda ^{-\frac{n+\alpha }{p}}\big\|\mathcal{M}(\varphi _{0}\ast
	f)\big\|_{L^{p}(\mathbb{R}^{n},|\cdot |^{\alpha })} \\
	&\lesssim \lambda ^{-\frac{n+\alpha }{p}}\big\|\varphi _{0}\ast f\big\|%
	_{L^{p}(\mathbb{R}^{n},|\cdot |^{\alpha })},
\end{align*}%
where the implicit constant is independent of $\lambda $. Now, let $v\in
\{0,k-1,k,k+1,k+2\}$. We have%
\begin{align*}
	\mathcal{F}^{-1}(\mathcal{F}\varphi _{0}\mathcal{F}(f(\lambda \cdot )))(x)
	&=\mathcal{F}^{-1}(\mathcal{F}\varphi _{0}(\lambda \cdot )\mathcal{F}%
	(f))(\lambda x) \\
	&=\mathcal{F}^{-1}(\mathcal{F}\varphi _{0}(\lambda \cdot )\mathcal{F}%
	\varphi _{0}\mathcal{F}(f))(\lambda x) \\
	&=\omega _{\lambda }\ast \varphi _{0}\ast f(\lambda x),\quad x\in \mathbb{R}%
	^{n},
\end{align*}%
where\ $\omega _{\lambda }=\lambda ^{-n}\psi (\lambda ^{-1}\cdot )$. We also
obtain%
\begin{align*}
	\mathcal{F}^{-1}(\mathcal{F}\varphi _{k-1}\mathcal{F}(f(\lambda \cdot )))(x)
	&=\mathcal{F}^{-1}(\mathcal{F}\varphi _{1}(\lambda 2^{2-k}\cdot )\mathcal{F}%
	(f))(\lambda x) \\
	&=\tilde{\varphi}_{\lambda ,k}\ast \varphi _{0}\ast f(\lambda x),\quad x\in 
	\mathbb{R}^{n},
\end{align*}%
where $\tilde{\varphi}_{\lambda ,k}=\lambda ^{-n}2^{kn-2n}\varphi
_{1}(\lambda ^{-1}2^{k-2}\cdot )$. If $l\in \{0,1,2\}$, then we obtain%
\begin{equation*}
\mathcal{F}^{-1}(\mathcal{F}\varphi _{k+l}\mathcal{F}(f(\lambda \cdot
)))(x)=\sum_{i=0}^{l+1}\tilde{\varphi}_{\lambda ,k+l+1}\ast \varphi _{i}\ast
f(\lambda x).
\end{equation*}%
Therefore,%
\begin{equation*}
\big\|\mathcal{F}^{-1}(\mathcal{F}\varphi _{v}\mathcal{F}(f(\lambda\cdot )))%
\big\|_{L^{p}(\mathbb{R}^{n},|\cdot |^{\alpha })}\lesssim \lambda ^{-\frac{%
		n+\alpha }{p}}\sum_{i=0}^{3}\big\|\varphi _{i}\ast f\big\|_{L^{p}(\mathbb{R}%
	^{n},|\cdot |^{\alpha })},
\end{equation*}%
where $v\in \{0,k-1,k,k+1,k+2\}$ and we have we used Lemmas \ref%
{Fefferman-Stein} and \ref{Marschall-ineq}. Using the estimate%
\begin{equation*}
\Big(\sum\limits_{i=0}^{\infty }a_{i}\Big)^{\delta }\leq
\sum\limits_{i=0}^{\infty }a_{i}^{\delta },\quad 0<\delta \leq 1,a_{i}\geq
0,i\in \mathbb{N}_{0},
\end{equation*}%
we get%
\begin{equation*}
\Big\|\Big(\sum\limits_{v=0}^{k+2}2^{vsq}|\mathcal{F}^{-1}(\mathcal{F}%
\varphi _{v}\mathcal{F}(f(\lambda\cdot )))|^{q}\Big)^{1/q}\Big\|_{L^{p}(%
	\mathbb{R}^{n},|\cdot |^{\alpha })}
\end{equation*}%
can be estimated by%
\begin{align*}
	&\sum\limits_{v=0}^{k+2}2^{vs}\Big\|\mathcal{F}^{-1}(\mathcal{F}\varphi _{v}%
	\mathcal{F}(f(\lambda \cdot )))\Big\|_{L^{p}(\mathbb{R}^{n},|\cdot |^{\alpha
		})} \\
	&\leq c\Big(\sum\limits_{v=1}^{k-1}2^{vs}+1\Big)\lambda ^{-\frac{n+\alpha }{%
			p}}\big\|\varphi _{0}\ast f\big\|_{L_{p}(\mathbb{R}^{n},|x|^{\alpha
		})}+c\lambda ^{-\frac{n+\alpha }{p}}\sum_{i=1}^{3}\big\|\varphi _{i}\ast f%
	\big\|_{L^{p}(\mathbb{R}^{n},|\cdot |^{\alpha })} \\
	&\leq c\lambda ^{s-\frac{n+\alpha }{p}}\big\|\varphi _{0}\ast f\big\|%
	_{L_{p}(\mathbb{R}^{n},|x|^{\alpha })}+c\lambda ^{s-\frac{n+\alpha }{p}%
	}\sum_{i=1}^{3}\big\|\varphi _{i}\ast f\big\|_{L^{p}(\mathbb{R}^{n},|\cdot
		|^{\alpha })} \\
	&\leq c\lambda ^{s-\frac{n+\alpha }{p}}\big\|f\big\|_{F_{p,q}^{s}(\mathbb{R}%
		^{n},|x|^{\alpha })},
\end{align*}%
since $s>0$, where the positive constant $c$ is independent of $k$. The
proof is complete.

\end{proof}

\begin{remark}
We would like mention that Theorem \ref{dilation} can be extended to $%
0<p<\infty ,0<q\leq \infty ,-n<\alpha <n(p-1)$ and\ $s>\max \big(0,\frac{n}{%
p }-n\big).$
\end{remark}

Let $k\in \mathbb{N},a,b\in \mathbb{R}$, $\beta >0$\ and $\lambda \geq 1+25%
\sqrt{n}$ be such that $a<b$. Let $j\in \mathbb{N},z^{j}=\big(%
z_{1}^{j},0,...,0\big)$ with $z_{1}^{1}=0$, $z_{1}^{j}=\frac{1}{4\sqrt{n}
j^{\lambda -1}},j\geq 2$.\ We set 
\begin{equation*}
P_{k}=\Big\{x:\frac{1}{16\sqrt{n}k^{\lambda }}\leq x_{i}\leq \frac{1}{2\sqrt{
n}k^{\lambda }},\ i=2,...,n,\ \frac{a}{k^{\lambda +\beta }}<x_{1}-z_{1}^{k}<%
\frac{b}{k^{\lambda +\beta }}\Big\}
\end{equation*}%
and 
\begin{equation*}
P_{k}^{\ast }=\Big\{x:\frac{1}{16\sqrt{n}k^{\lambda }}\leq x_{i}\leq \frac{1%
}{8\sqrt{n}k^{\lambda }},\ i=2,...,n,\ \frac{a}{k^{\lambda +\beta }}%
<x_{1}-z_{1}^{k}<\frac{a+b}{2k^{\lambda +\beta }}\Big\}.
\end{equation*}%
The following lemma is very important for what will follow.

\begin{lemma}
\label{Dah}Let $\gamma =\lfloor \big(4\sqrt{n}(|b|+|a|)\big)^{\frac{1}{\beta 
}}\rfloor +3$. $\newline
\mathrm{(i)}$ There exists $v\in \mathbb{N}$ such that 
\begin{equation*}
P_{k}\cap P_{r}=\emptyset \quad \text{if}\quad r\neq k,\quad r,k\geq \max
(v,\gamma ).
\end{equation*}
$\mathrm{(ii)}$ Let $0<t<\frac{T}{8k^{\lambda +\beta }},T=b-a$. If $k\geq
\gamma $ then 
\begin{equation*}
x\in P_{k}^{\ast }\quad \text{implies}\quad x+\ell h\in P_{k}
\end{equation*}
for any $\ell =0,1,2,h=(h_{1},...,h_{n}),0<h_{i}<\frac{t}{\sqrt{n}}
,i=1,...,n $.
\end{lemma}

\begin{proof}
We will do the proof in two steps.

\textbf{Step\ 1.}\ Proof of\ (i). Define%
	\begin{equation*}
	U_{j}=j-(j+1)\big(\frac{j}{j+1}\big)^{\lambda },\quad j\in \mathbb{N}.
	\end{equation*}%
	Let $v\in \mathbb{N}$ be such that%
	\begin{equation}
	U_{j}-4\sqrt{n}\big(\frac{j}{j+1}\big)^{\lambda }\geq 20\sqrt{n},\quad j\geq
	v,  \label{Estimate-h}
	\end{equation}%
	which is possible since 
	\begin{equation*}
	\lim_{j\longrightarrow \infty }\big(U_{j}-4\sqrt{n}\big(\frac{j}{j+1}\big)%
	^{\lambda }\big)=\lambda -4\sqrt{n}-1\geq21\sqrt{n}.
	\end{equation*}%
	We claim that 
	\begin{equation}
	x\in P_{k}\quad \text{implies}\quad |x-z^{k}|<\frac{1}{k^{\lambda }}.
	\label{pro-P}
	\end{equation}%
	Indeed, observe that 
	\begin{equation*}
	\frac{-1}{4\sqrt{n}k^{\lambda }}<\frac{a}{k^{\lambda +\beta }}%
	<x_{1}-z_{1}^{k}<\frac{b}{k^{\lambda +\beta }}<\frac{1}{4\sqrt{n}k^{\lambda }%
	}.
	\end{equation*}%
	Let $k,r\geq \max (v,\gamma )$. It is obvious that 
	\begin{equation*}
	4\sqrt{n}(z_{1}^{k}-z_{1}^{r})=\frac{1}{k^{\lambda -1}}-\frac{1}{r^{\lambda
			-1}}.
	\end{equation*}%
	If $r>k$, then 
	\begin{equation*}
	\frac{1}{k^{\lambda -1}}-\frac{1}{r^{\lambda -1}}\geq \frac{1}{k^{\lambda -1}%
	}-\frac{1}{(k+1)^{\lambda -1}}=\frac{1}{k^{\lambda }}\big(k-(k+1)\big(\frac{k%
	}{k+1}\big)^{\lambda }\big)=\frac{1}{k^{\lambda }}U_{k}.
	\end{equation*}%
	Let $x\in P_{r}$, with $r>k$. The triangle inequality gives 
	\begin{equation*}
	|x_{1}-z_{1}^{k}|\geq |z_{1}^{k}-z_{1}^{r}|-|x_{1}-z_{1}^{r}|\geq
	|z_{1}^{k}-z_{1}^{r}|-\frac{1}{r^{\lambda }}\geq \frac{1}{4\sqrt{n}%
		k^{\lambda }}U_{k}-\frac{1}{(k+1)^{\lambda }}.
	\end{equation*}%
	But%
	\begin{equation*}
	\frac{1}{4\sqrt{n}k^{\lambda }}U_{k}-\frac{1}{(k+1)^{\lambda }}=\frac{1}{4%
		\sqrt{n}k^{\lambda }}\big(U_{k}-4\sqrt{n}\big(\frac{k}{k+1}\big)^{\lambda }%
	\big).
	\end{equation*}%
	Thanks to $\mathrm{\eqref{Estimate-h}}$\ we end up with $|x_{1}-z_{1}^{k}|>%
	\frac{5}{k^{\lambda }}$\ for any $k\geq \max (v,\gamma )$. That gives\textrm{%
		,} 
	\begin{equation*}
	P_{k}\cap P_{r}=\emptyset \quad \text{if}\quad r>k\geq \max (v,\gamma ).
	\end{equation*}%
	Now, let $x\in P_{r}$ with $\max (v,\gamma )\leq r<k$. Again by the triangle
	inequality, we obtain 
	\begin{align*}
		|x_{1}-z_{1}^{k}| &\geq |z_{1}^{k}-z_{1}^{r}|-|x_{1}-z_{1}^{r}|\geq
		|z_{1}^{k}-z_{1}^{r}|-\frac{1}{r^{\lambda }} \\
		&\geq \frac{1}{4\sqrt{n}}\Big(\frac{1}{r^{\lambda -1}}-\frac{1}{%
			(r+1)^{\lambda -1}}\Big)-\frac{1}{r^{\lambda }} \\
		&=\frac{1}{4\sqrt{n}r^{\lambda }}U_{r}-\frac{1}{r^{\lambda }},
	\end{align*}%
	which yields  $|x_{1}-z_{1}^{k}|>\frac{5}{r^{\lambda }}\geq \frac{5}{%
		k^{\lambda }}$ for any $k>r\geq \max (v,\gamma )$ where we used again $%
	\mathrm{\eqref{Estimate-h}}$. This gives the desired result.
	 
	\textbf{Step\ 2.}\ Proof of\ (ii). Let $ i\in \{2,...,n\}$ and $x\in P_{k}^{\ast }$. We have%
	\begin{equation*}
	\frac{1}{16\sqrt{n}k^{\lambda }}\leq \frac{1}{16\sqrt{n}k^{\lambda }}+\ell
	h_{i}\leq x_{i}+\ell h_{i}\leq \frac{1}{8\sqrt{n}k^{\lambda }}+\ell h_{i}<%
	\frac{1}{8\sqrt{n}k^{\lambda }}+\frac{\ell t}{\sqrt{n}}.
	\end{equation*}
		Since $k\geq \gamma $, w obtain%
	\begin{equation*}
	k^{\beta }\geq 4\sqrt{n}(|b|+|a|)\geq 4\sqrt{n}T.
	\end{equation*}%
	This together with $0<t<\frac{T}{8k^{\lambda +\beta }}$ yields that 
	\begin{equation*}
	\frac{1}{16\sqrt{n}k^{\lambda }}\leq x_{i}+\ell h_{i}<\frac{1}{8\sqrt{n}%
		k^{\lambda }}+\frac{1}{4\sqrt{n}k^{\lambda }}\frac{T}{k^{\beta }}\leq \frac{1%
	}{2\sqrt{n}k^{\lambda }},\quad i=2,...,n.
	\end{equation*}%
	Now,%
	\begin{equation*}
	\frac{a}{k^{\lambda +\beta }}<x_{1}-z_{1}^{k}+\ell h_{1}<\frac{a+b}{%
		2k^{\lambda +\beta }}+\frac{2t}{\sqrt{n}}<\frac{a+b}{2k^{\lambda +\beta }}+%
	\frac{T}{2k^{\lambda +\beta }}\leq \frac{b}{k^{\lambda +\beta }},
	\end{equation*}%
	since $T=b-a$. The proof is complete.
	
\end{proof}

\section{Proof of the results}

The main aim of this section is to present the proof of Theorems \ref%
{Triviality1}-\ref{Triviality1 copy(1)}\textbf{. } We follow the same
notations as in \cite[Chapter 5]{RS96}.

\begin{definition}
$\mathrm{(i)}$ For $f\in \mathcal{S}^{\prime }(\mathbb{R}^{n})$ we define a
distribution $\bar{f}$ by 
\begin{equation*}
\bar{f}(\varphi )=\overline{f(\bar{\varphi})},\quad \varphi \in \mathcal{S}( 
\mathbb{R}^{n}).
\end{equation*}
$\mathrm{(ii)}$ The space of real-valued distributions $\mathbb{S}^{\prime
}( \mathbb{R}^{n})$ is defined to be 
\begin{equation*}
\mathbb{S}^{\prime }(\mathbb{R}^{n})=\{f\in \mathcal{S}^{\prime }(\mathbb{R}
^{n}):\bar{f}=f\}.
\end{equation*}
$\mathrm{(iii)}$ Let $A$ be a complex-valued, normed distribution space such
that $A\hookrightarrow \mathcal{S}^{\prime }(\mathbb{R}^{n})$. Then we
define the real-valued part $\mathbb{A}$ by $\mathbb{A=}A\cap $ $\mathbb{S}
^{\prime }(\mathbb{R}^{n})$ equipped with the same norm as $A$.
\end{definition}

We set%
\begin{equation*}
\mathbb{F}_{p,q}^{s}(\mathbb{R}^{n},|\cdot |^{\alpha })=\{f\in F_{p,q}^{s}(%
\mathbb{R}^{n},|\cdot |^{\alpha }):f\text{ is real-valued}\},
\end{equation*}%
\begin{equation*}
\mathbb{B}_{p,q}^{s}(\mathbb{R}^{n},|\cdot |^{\alpha })=\{f\in B_{p,q}^{s}(%
\mathbb{R}^{n},|\cdot |^{\alpha }):f\text{ is real-valued}\}.
\end{equation*}

\begin{proof}[\textbf{\ Proof of Theorem \protect\ref{Triviality1}}]
	The proof is based on \cite[Theorem 1]{Si97} and  \cite[Theorem 5.3.1/3]%
	{RS96}. By the embeddings $A_{p,q}^{s}(\mathbb{R}^{n},|\cdot |^{\alpha
	})\hookrightarrow B_{p,\infty }^{s}(\mathbb{R}^{n},|\cdot |^{\alpha })$ and $%
	B_{p,1}^{s}(\mathbb{R}^{n},|\cdot |^{\alpha })\hookrightarrow A_{p,q}^{s}(%
	\mathbb{R}^{n},|\cdot |^{\alpha })$, we need only prove that any composition
	operator that takes $B_{p,1}^{s}(\mathbb{R}^{n},|\cdot |^{\alpha })$ to $%
	B_{p,\infty }^{s}(\mathbb{R}^{n},|\cdot |^{\alpha })$ is linear.
	
	\textbf{Step 1. }First we need to construct a suitable function. Assume that 
	$G$ is not of the form $\mathrm{\eqref{linear}}$. Since $G\in C^{2}(\mathbb{R%
	})$ there exist two real numbers $a$ and $b$, and a number $C>0$ such that%
	\begin{equation*}
	|G^{(2)}(x)|\geq C\quad x\in I=[a,b].
	\end{equation*}%
	Using the mean-value theorem for higher order differences, see \cite[\S 60]%
	{He90} we obtain%
	\begin{equation}
	|(\Delta _{h}^{2}G)(x)|=|G^{(2)}(\xi )||h|^{2},\quad h\neq 0,\quad \xi \in
	]x,x+2h[,  \label{est-G-1}
	\end{equation}%
	which yields that%
	\begin{equation}
	|(\Delta _{h}^{2}G)(x)|\geq C|h|^{2}\quad \text{if}\quad x,x+2h\in I.
	\label{est-G-2}
	\end{equation}%
	Let $\theta ,\eta \in \mathcal{S}(\mathbb{R}^{n})$ be compactly supported
	positive functions such that%
	\begin{equation*}
	\text{supp}\theta \subset \big\{x:|x_{i}|\leq 4,\quad i=1,2,...,n\big\},
	\end{equation*}%
	\begin{equation*}
	\theta (x)=1,\quad \text{if}\quad |x_{i}|\leq 2,\quad i=1,2,...,n,
	\end{equation*}%
	\begin{equation*}
	\text{supp}\eta \subset \big\{x:|x_{i}|\leq 1,\quad i=1,2,...,n\big\},\quad %
	\big\|\eta \big\|_{1}=1
	\end{equation*}%
	and%
	\begin{equation*}
	\eta (x)=1,\quad \text{if}\quad |x_{i}|\leq \frac{1}{2},\quad i=1,2,...,n.
	\end{equation*}%
	Let $\lambda >1,j\in \mathbb{N},z^{j}=\big(z_{1}^{j},z_{2}^{j},...,z_{n}^{j}%
	\big)$ with $z_{1}^{1}=0$, $z_{1}^{j}=\frac{1}{4\sqrt{n}j^{\lambda -1}}%
	,j\geq 2$ and $z_{2}^{j}=\cdot \cdot \cdot =z_{n}^{j}=0,j\in \mathbb{N}$.
	Consider%
	\begin{equation*}
	g_{j}(x)=j^{\lambda }(x_{1}-z_{1}^{j})\theta _{j}\ast \breve{\eta}_{j}(x),
	\end{equation*}%
	where $x\in \mathbb{R}^{n},j\in \mathbb{N},\theta _{j}=\theta (j^{\lambda
	}\cdot )$, $\breve{\eta}_{j}=\frac{\eta _{j}}{\big\|\eta _{j}\big\|_{1}},$%
	\begin{equation*}
	\eta _{j}=\eta (j^{\lambda }(\cdot -z^{j})),\quad j\in \mathbb{N}.
	\end{equation*}%
	We have%
	\begin{equation*}
	\text{supp}\theta _{j}\ast \eta _{j}\subset \big\{x:|x_{i}-z_{i}^{j}|\leq
	5j^{-\lambda },\quad i=1,2,...,n\big\}
	\end{equation*}%
	and%
	\begin{equation*}
	\theta _{j}\ast \eta _{j}\equiv 1
	\end{equation*}%
	on%
	\begin{equation*}
	\big\{x:|x_{i}-z_{i}^{j}|\leq j^{-\lambda },\quad i=1,2,...,n\big\}.
	\end{equation*}%
	Thanks to Theorem \ref{dilation} there exists a positive constant $c$ such
	that%
	\begin{align*}
		\big\|g_{j}\big\|_{B_{p,1}^{s}(\mathbb{R}^{n},|\cdot |^{\alpha })} &\leq
		cj^{\lambda (s-\frac{n+\alpha }{p})}\big\|g_{j}(j^{-\lambda }\cdot )\big\|%
		_{B_{p,1}^{s}(\mathbb{R}^{n},|\cdot |^{\alpha })} \\
		&=cj^{\lambda (s-\frac{n+\alpha }{p})}\big\|\vartheta _{j}\big\|%
		_{B_{p,1}^{s}(\mathbb{R}^{n},|\cdot |^{\alpha })}
	\end{align*}%
	for any $j\in \mathbb{N}$, where the positive constant $c$ is independent of 
	$j$. Here 
	\begin{equation*}
	\vartheta _{j}(x)=(x_{1}-j^{\lambda }z_{1}^{j})\theta \ast \tau _{j^{\lambda
		}z^{j}}\eta (x),\quad x\in \mathbb{R}^{n}.
	\end{equation*}%
	Obviously $\vartheta _{j}=\omega \ast \tau _{j^{\lambda }z^{j}}\eta +\theta
	\ast \tau _{j^{\lambda }z^{j}}\tilde{\eta},$ where $\tilde{\eta}%
	(x)=x_{1}\eta (x)$, with $\omega (x)=x_{1}\theta (x),x\in \mathbb{R}^{n}$.
	Let $\{\mathcal{F}\varphi _{l}\}_{l\in \mathbb{N}_{0}}$ be a partition of
	unity. We claim that%
	\begin{equation}
	|\varphi _{l}\ast \omega \ast \tau _{j^{\lambda }z^{j}}\eta |\lesssim j^{L}%
	\mathcal{M}(\varphi _{l}\ast \omega )  \label{maximal2}
	\end{equation}%
	and%
	\begin{equation*}
	|\varphi _{l}\ast \theta \ast \tau _{j^{\lambda }z^{j}}\tilde{\eta}|\lesssim
	j^{L}\mathcal{M}(\varphi _{l}\ast \theta ),
	\end{equation*}%
	where $L>13n$ and the implicit constant is independent of $l$ and $j$. Hence%
	\begin{equation*}
	\big\|g_{j}\big\|_{B_{p,1}^{s}(\mathbb{R}^{n},|\cdot |^{\alpha })}\lesssim
	j^{\lambda (s-\frac{n+\alpha }{p})+L}.
	\end{equation*}%
	We prove our claim. By similarity, we prove only \eqref{maximal2}. Let $%
	x,y\in \mathbb{R}^{n}\ $and $j\in \mathbb{N}$. Since $\eta $ is a Schwartz
	function, then%
	\begin{align*}
		|\tau _{j^{\lambda }z^{j}}\eta (x-y)| &\leq c(1+|x-y-j^{\lambda
		}z^{j}|)^{-L} \\
		&\leq c(1+|x-y|)^{-L}(1+j^{\lambda }|z^{j}|)^{L} \\
		&\leq cj^{L}(1+|x-y|)^{-L}
	\end{align*}%
	where the positive constant $c$ is indepedent of $l,x,y$ and $j$. We set $%
	\varpi _{L}(x)=(1+|x|)^{-L},$ $x\in \mathbb{R}^{n}$. Hence%
	\begin{align}
	&|\varphi _{l}\ast \omega \ast \tau _{j^{\lambda }z^{j}}\eta (x)|  \notag \\
	&\lesssim \int_{\mathbb{R}^{n}}|\varphi _{l}\ast \omega (y)||\tau
	_{j^{\lambda }z^{j}}\eta (x-y)|dy  \notag \\
	&\lesssim j^{L}\varpi _{L}\ast |\varphi _{l}\ast \omega |(x)  \notag \\
	&\lesssim j^{L}\varpi _{L}\chi _{B(x,2)}\ast |\varphi _{l}\ast \omega
	|(x)+j^{L}\varpi _{L}\chi _{\mathbb{R}^{n}\backslash B(x,2)}\ast |\varphi
	_{l}\ast \omega |(x).  \label{estimate-f1.2}
	\end{align}%
	Obviously, the first term of \eqref{estimate-f1.2} is bounded by $c\mathcal{M%
	}(\varphi _{l}\ast \omega )(x)$. We have%
	\begin{align*}
		\varpi _{L}\chi _{\mathbb{R}^{n}\backslash B(x,2)}\ast |\varphi _{l}\ast
		\omega |(x) &=\sum_{i=1}^{\infty }\varpi _{L}\chi _{B(x,2^{i+1})\backslash
			B(x,2^{i})}\ast |\varphi _{l}\ast \omega |(x) \\
		&\leq \sum_{i=1}^{\infty }2^{-iL}\varpi _{L}\chi _{B(x,2^{i+1})}\ast
		|\varphi _{l}\ast \omega |(x) \\
		&\lesssim \mathcal{M}(\varphi _{l}\ast \omega )(x)\sum_{i=1}^{\infty
		}2^{i(n-L)} \\
		&\lesssim \mathcal{M}(\varphi _{l}\ast \omega )(x).
	\end{align*}
	Assume that $\lambda \geq 1+25\sqrt{n}$. Let $\beta >0,\gamma =\lfloor \big(4%
	\sqrt{n}(|b|+|a|)\big)^{\frac{1}{\beta }}\rfloor +3$ and%
	\begin{equation*}
	f(x)=\sum_{j=\max (v,\gamma )}^{\infty }j^{\beta }g_{j}(x),\quad x\in 
	\mathbb{R}^{n}.
	\end{equation*}%
	Let us prove that $f$ belongs to $B_{p,1}^{s}(\mathbb{R}^{n},|\cdot
	|^{\alpha })$. Obviously%
	\begin{align*}
		\big\|f\big\|_{B_{p,1}^{s}(\mathbb{R}^{n},|\cdot |^{\alpha })} &\leq
		c\sum_{j=\max (v,\gamma )}^{\infty }j^{\beta }\big\|g_{j}\big\|%
		_{B_{p,1}^{s}(\mathbb{R}^{n},|\cdot |^{\alpha })} \\
		&\leq c\sum_{j=\max (v,\gamma )}^{\infty }j^{\beta +L+\lambda (s-\frac{%
				n+\alpha }{p})}.
	\end{align*}%
	If%
	\begin{equation}
	\beta +L+\lambda \big(s-\frac{n+\alpha }{p}\big)<-1,
	\label{f-belonging in B}
	\end{equation}%
	then we have $f\in B_{p,1}^{s}(\mathbb{R}^{n},|\cdot |^{\alpha })$.
	
	\textbf{Step 2. }Let 
	\begin{equation}
	\max \Big(p,\frac{n+\alpha }{2-s+\frac{n+\alpha }{p}}\Big)<p_{1}<\frac{%
		n+\alpha -1}{\frac{n+\alpha }{p}-s+1}\quad \text{and}\quad s_{1}=s-\frac{%
		n+\alpha }{p}+\frac{n+\alpha }{p_{1}},  \label{assumption1}
	\end{equation}%
	which is possible, since $s>1+\frac{1}{p}$. That choice guarantees%
	\begin{equation*}
	p<p_{1},\quad 0<s_{1}<2\quad \text{and}\quad B_{p,\infty }^{s}(\mathbb{R}%
	^{n},|\cdot |^{\alpha })\hookrightarrow B_{p_{1},\infty }^{s_{1}}(\mathbb{R}%
	^{n},|\cdot |^{\alpha }),
	\end{equation*}%
	by Theorem \ref{embeddings5}. We will prove that 
	\begin{equation}
	T_{G}(f)\notin B_{p_{1},\infty }^{s_{1}}(\mathbb{R}^{n},|\cdot |^{\alpha }).
	\label{Case1}
	\end{equation}%
	As in \cite[Theorem 1]{Si97}, see also \cite[Theorem 5.3.1/3]{RS96}, we use
	the two sequences of cubes $P_{k}$ and $P_{k}^{\ast },k\geq \max (v,\gamma )$%
	. From Lemma \ref{Dah} and $\mathrm{\eqref{pro-P}}$, we get 
	\begin{equation}
	f(y)=k^{\beta }g_{k}(y)=k^{\beta +\lambda }(y_{1}-z_{1}^{k})  \label{key}
	\end{equation}%
	for any $y\in P_{k}$ and any $k\geq \max (v,\gamma )$. Let%
	\begin{equation*}
	\varrho (t)=\Big(\frac{T}{8t}\Big)^{\frac{1}{\lambda +\beta }},\quad
	t>0,\quad T=b-a
	\end{equation*}%
	and $0<t<\frac{T}{8k^{\lambda +\beta }}$. Again, from Lemma \ref{Dah} if $%
	\max (v,\gamma )\leq k\leq \varrho (t)$ then 
	\begin{equation}
	x\in P_{k}^{\ast }\quad \text{implies}\quad x+\ell h\in P_{k}  \label{cube}
	\end{equation}%
	for any $\ell =0,1,2,h=(h_{1},...,h_{n}),0<h_{i}<\frac{t}{\sqrt{n}},i=1,...,n
	$. We have 
	\begin{equation*}
	|\Delta _{h}^{2}(T_{G}(f))(x)|=\big|%
	\sum_{i=0}^{2}C_{i}^{2}(-1)^{i}G(f(x+(2-i)h))\big|,
	\end{equation*}%
	where $C_{i}^{2},i\in \{0,1,2\}$\ are the binomial coefficients. Let $x\in
	P_{k}^{\ast }$. From $\mathrm{\eqref{key}}$ and $\mathrm{\eqref{cube}}$,
	with the help of $\mathrm{\eqref{est-G-1}}$ and $\mathrm{\eqref{est-G-2}}$,
	we obtain%
	\begin{align*}
		|\Delta _{h}^{2}(T_{G}(f))(x)| &=\big|\sum_{i=0}^{2}C_{i}^{2}(-1)^{i}G(k^{%
			\beta +\lambda }(x_{1}+(2-i)h_{1}-z_{1}^{k}))\big| \\
		&=|\Delta _{k^{\lambda +\beta }h_{1}}^{2}G(k^{\lambda +\beta
		}(x_{1}-z_{1}^{k}))| \\
		&\geq C|h_{1}|^{2}k^{2(\lambda +\beta )},
	\end{align*}%
	where the positive constant $C$ is independent of $k$ and $h$. Consequently%
	\begin{equation*}
	d_{t}^{2}(T_{G}(f))(x)=t^{-n}\int_{|h|\leq t}\left\vert \Delta
	_{h}^{2}(T_{G}(f))(x)\right\vert dh\geq ct^{2}k^{2(\lambda +\beta )},
	\end{equation*}%
	which yields that 
	\begin{equation*}
	\big\|d_{t}^{2}(T_{G}(f))\big\|_{L^{p_{1}}(\mathbb{R}^{n},|\cdot |^{\alpha
		})}^{p_{1}}
	\end{equation*}%
	is greater than%
	\begin{align*}
		ct^{2p_{1}}\sum_{k=\max (v,\gamma )}^{\lfloor \varrho (t)\rfloor
		}k^{2(\lambda +\beta )p_{1}}\int_{P_{k}^{\ast }}|x|^{\alpha }dx &\geq
		ct^{2p_{1}}\sum_{k=\max (v,\gamma )}^{\lfloor \varrho (t)\rfloor
		}k^{2(\lambda +\beta )p_{1}-\lambda \alpha }\int_{P_{k}^{\ast }}dx \\
		&\geq ct^{2p_{1}}\sum_{k=\lfloor \frac{1}{4}\varrho (t)\rfloor +1}^{\lfloor
			\varrho (t)\rfloor }k^{(\lambda +\beta )(2p_{1}-1)-\lambda (n-1+\alpha )}
	\end{align*}%
	for any $t>0$ sufficiently small, with the help of the fact that $\alpha
	\geq 0$. Therefore%
	\begin{equation*}
	\big\|T_{G}(f)\big\|_{B_{p_{1},\infty }^{s_{1}}(\mathbb{R}^{n},|\cdot
		|^{\alpha })}^{p_{1}}\geq c\sup_{t>0}t^{(2-s_{1})p_{1}}\sum_{k=\lfloor \frac{%
			1}{4}\varrho (t)\rfloor +1}^{\lfloor \varrho (t)\rfloor }k^{(\lambda +\beta
		)(2p_{1}-1)-\lambda (n-1+\alpha )}.
	\end{equation*}%
	We claim that%
	\begin{equation*}
	\sum_{k=\lfloor \frac{1}{4}\varrho (t)\rfloor +1}^{\lfloor \varrho
		(t)\rfloor }k^{(\lambda +\beta )(2p_{1}-1)-\lambda (n-1+\alpha )}\approx t^{-%
		\frac{1}{\lambda +\beta }\big((\lambda +\beta )(2p_{1}-1)-\lambda
		(n-1+\alpha )+1\big)},
	\end{equation*}%
	which yields $\mathrm{\eqref{Case1}}$ if $\frac{\lambda (n-1+\alpha )-1}{%
		(\lambda +\beta )p_{1}}<s_{1}-\frac{1}{p_{1}}$ provided $\beta $ fulfills $%
	\mathrm{\eqref{f-belonging in B}}$. We may choose%
	\begin{equation*}
	\max \Big(1+25\sqrt{n},\frac{\beta +1+L}{\frac{n+\alpha }{p}-s}\Big)<\lambda
	<\frac{\frac{1}{p_{1}}+\beta \big(s_{1}-\frac{1}{p_{1}}\big)}{\frac{n+\alpha 
		}{p}-s}.
	\end{equation*}%
	Our assumptions leads to $\mathrm{\eqref{Case1}}$ by taking%
	\begin{equation*}
	\beta >\max \Big(0,\frac{(1+25\sqrt{n})\big(\frac{n+\alpha }{p}-s\big)-\frac{%
			1}{p_{1}}}{s_{1}-\frac{1}{p_{1}}},\frac{\frac{1}{p_{1}}-1-L}{1-s_{1}+\frac{1%
		}{p_{1}}}\Big),
	\end{equation*}%
	which is possible in view of $\mathrm{\eqref{assumption1}}$. Since $G(0)=0$
	is necessary for $\mathrm{\eqref{Condition2}}$, it follows that\ $%
	G(t)=ct,t\in \mathbb{R}$ for some constant $c$. We prove our claim. Let $%
	k\in \mathbb{N}$ be such that%
	\begin{equation*}
	\lfloor \frac{1}{4}\varrho (t)\rfloor +1\leq k\leq \lfloor \varrho
	(t)\rfloor ,\quad t>0,
	\end{equation*}%
	which yields that 
	\begin{equation*}
	\frac{1}{4}\varrho (t)\leq k\leq \varrho (t).
	\end{equation*}%
	Thus%
	\begin{align*}
		\sum_{k=\lfloor \frac{1}{4}\varrho (t)\rfloor +1}^{\lfloor \varrho
			(t)\rfloor }k^{(\lambda +\beta )(2p_{1}-1)-\lambda (n-1+\alpha )} &\approx
		\sum_{k=\lfloor \frac{1}{4}\varrho (t)\rfloor +1}^{\lfloor \varrho
			(t)\rfloor }(\varrho (t))^{(\lambda +\beta )(2p_{1}-1)-\lambda (n-1+\alpha )}
		\\
		&\approx t^{-\big((2p_{1}-1)-\frac{\lambda (n-1+\alpha )}{\lambda +\beta }%
			\big)}\sum_{k=\lfloor \frac{1}{4}\varrho (t)\rfloor +1}^{\lfloor \varrho
			(t)\rfloor }1 \\
		&\approx t^{-\big((2p_{1}-1)-\frac{\lambda (n-1+\alpha )}{\lambda +\beta }%
			\big)}(\lfloor \varrho (t)\rfloor -\lfloor \frac{1}{4}\varrho (t)\rfloor ),
	\end{align*}%
	since $\varrho (t)=\Big(\frac{T}{8t}\Big)^{\frac{1}{\lambda +\beta }}$.
	Observe that%
	\begin{equation*}
	\lfloor \varrho (t)\rfloor -\lfloor \frac{1}{4}\varrho (t)\rfloor >\frac{3}{4%
	}\varrho (t)-1\geq \frac{1}{2}\varrho (t)
	\end{equation*}%
	and%
	\begin{equation*}
	\lfloor \varrho (t)\rfloor -\lfloor \frac{1}{4}\varrho (t)\rfloor <\frac{3}{4%
	}\varrho (t)+1\leq \varrho (t)
	\end{equation*}%
	for sufficiently small $t>0$. Therefore our claim is proved. The proof is
	complete.
\end{proof}

From \eqref{Sobolev}, we immediately obtain the following statement.

\begin{corollary}
Let\ $1<p<\infty ,m\in \mathbb{N}$ and $0\leq \alpha <n(p-1)$. Let $G\in
C^{2}(\mathbb{R})$ and $T_{G}$ be a composition operator.\ Suppose 
\begin{equation*}
2\leq m<\frac{n+\alpha }{p}
\end{equation*}
and the acting condition 
\begin{equation*}
T_{G}(\mathbb{W}_{p}^{m}(\mathbb{R}^{n},|\cdot |^{\alpha }))\subset \mathbb{%
W }_{p}^{m}(\mathbb{R}^{n},|\cdot |^{\alpha }).
\end{equation*}
Then 
\begin{equation*}
G(t)=ct,\quad t\in \mathbb{R}
\end{equation*}
for some constant $c$.
\end{corollary}

To prove Theorem \ref{Triviality1 copy(1)}, we need some preparations.
Consider the partition of the unity $\{\mathcal{F}\varphi _{j}\}_{j\in 
\mathbb{N}_{0}}$. We define the convolution operators $\Lambda _{j}$\ by the
following: 
\begin{equation*}
\Lambda _{j}f=\varphi _{j}\ast f,\quad j\in \mathbb{N}\quad \text{and}\quad
\Lambda _{0}f=\mathcal{F}^{-1}\psi \ast f,\quad f\in \mathcal{S}^{\prime }(%
\mathbb{R}^{n}).
\end{equation*}%
We associate the convolution operator $Q_{j}$ defined as $Q_{j}f=\mathcal{F}%
^{-1}(\psi (2^{-j}\cdot ))\ast f,j\in \mathbb{N}$. We set $\Lambda
_{0}=Q_{0},$ thus we obtain $Q_{j}f=\sum_{i=0}^{j+1}\Lambda _{i}f,j\in 
\mathbb{N}$. For $f\in \mathcal{S}^{\prime }(\mathbb{R}^{n})$ we define the
product of these two distributions as%
\begin{equation}
T_{G}(f)=f^{2}=\lim_{j\rightarrow \infty }Q_{j}f\cdot Q_{j}f
\label{multiplication}
\end{equation}%
whenever this limit exists in $\mathcal{S}^{\prime }(\mathbb{R}^{n})$.
Related to this definition we introduce the following operators%
\begin{equation*}
\Pi _{1}(f)=\sum_{j=2}^{\infty }Q_{j-2}f\Lambda _{j}f,\quad \Pi
_{2}(f)=\sum_{j=0}^{\infty }\overline{\Lambda }_{j}f\Lambda _{j}f
\end{equation*}%
and $\Pi _{3}(f)=\Pi _{1}(f)$, with $\overline{\Lambda }_{j}=%
\sum_{k=j-1}^{j+1}\Lambda _{k},j\in \mathbb{N}_{0}$. The advantage of the
above decomposition consists in%
\begin{equation*}
\text{supp}\mathcal{F}(Q_{j-2}f\Lambda _{j}f)\subset \{\xi :2^{j-3}\leq |\xi
|\leq 2^{j+1}\},\quad j=2,3,...
\end{equation*}%
and%
\begin{equation*}
\text{supp}\mathcal{F}(\overline{\Lambda }_{j}f\Lambda _{j}f)\subset \{\xi
:|\xi |\leq 5\cdot 2^{j}\},\quad j\in \mathbb{N}_{0}.
\end{equation*}

The next lemma shows that the multiplication on the right-hand side of %
\eqref{multiplication} makes sense, but under some suitable assumptions.

\begin{lemma}
Let\ $1< p<\infty ,1\leq q\leq \infty ,0\leq \alpha <n(p-1)$ and $s>0$. Let $%
f\in A_{p,q}^{s}(\mathbb{R}^{n},|x|^{\alpha })$. Then $f^{2}$ is well
defined and $(Q_{j}f)^{2}$ tends to $f^{2}$ as $j$ tends to infinity.
\end{lemma}

\begin{proof}
From Theorem \ref{embeddings} the elements of $A_{p,q}^{s}(\mathbb{R}%
^{n},|x|^{\alpha })$ are regular distributions. So $f^{2}$ is well defined,
as an element of $L_{\mathrm{loc}}^{1}(\mathbb{R}^{n})$. Let us prove that $%
\{Q_{j}f\}_{j\in \mathbb{N}}$ converges to $f$ almost everywhere. Using $s>0$%
, we get 
\begin{equation*}
\sum_{i=0}^{\infty }\big\Vert\Lambda _{i}f\big\Vert_{L^{p}(\mathbb{R}%
^{n},|\cdot |^{\alpha })}\lesssim \big\Vert f\big\Vert_{A_{p,q}^{s}(\mathbb{R%
}^{n},|x|^{\alpha })}.
\end{equation*}%
Then the series $\sum_{i=0}^{\infty }\Lambda _{i}f$ converges in $L^{p}(%
\mathbb{R}^{n},|\cdot |^{\alpha })$ to a limit $g\in L^{p}(\mathbb{R}%
^{n},|\cdot |^{\alpha })$. Therefore $\{Q_{j}f\}_{j\in \mathbb{N}}$
converges to $g$ almost everywhere. Let $\varphi \in \mathcal{S}(\mathbb{R}%
^{n})$. We write%
\begin{equation*}
\langle f-g,\varphi \rangle =\langle f-Q_{j}f,\varphi \rangle +\langle
g-Q_{j}f,\varphi \rangle ,\quad j\in \mathbb{N}.
\end{equation*}%
Here $\langle \cdot ,\cdot \rangle $ denotes the duality bracket between $%
\mathcal{S}^{\prime }(\mathbb{R}^{n})$ and $\mathcal{S}(\mathbb{R}^{n})$.
The first term tends to zero as $j\rightarrow \infty $, while by H\"{o}%
lder's\ inequality there exists a constant $C>0$ independent of $j$ such
that 
\begin{equation*}
|\langle g-Q_{j}f,\varphi \rangle |\leq C\big\|g-Q_{j}f\big\|_{L^{p}(\mathbb{%
R}^{n},|\cdot |^{\alpha })},
\end{equation*}%
which tends to zero as $j\rightarrow \infty $. Therefore $f=g$ almost
everywhere. Consequently, $(Q_{j}f)^{2}$ tends to $f^{2}$ as $j$ tends to
infinity. The proof is complete.
 \end{proof}

\begin{remark}
For interested readers, we refer to \cite{J96} and \cite[Chapter 5]{RS96}
for more detailed discussion of \eqref{multiplication}.
\end{remark}

The next lemma is used in the proof of our Theorem \ref{Triviality1 copy(1)}%
, see \cite{BD} for the proof.

\begin{lemma}
\label{Key-lemma1}Let\textit{\ }$A,B>0,1<p<\infty ,1\leq q\leq \infty ,0\leq
\alpha <n(p-1)$ and $s>0$. Let $\left\{ f_{l}\right\} _{l\in \mathbb{N}_{0}}$
be a sequence of functions\ such that 
\begin{equation*}
\mathrm{supp}\mathcal{F}f_{l}\subseteq \left\{ \xi \in \mathbb{R}%
^{n}:\left\vert \xi \right\vert \leq A2^{l+1}\right\} ,\quad l\in \mathbb{N}%
_{0}.
\end{equation*}%
\textit{T}hen 
\begin{equation*}
\Big\|\sum_{l=0}^{\infty }f_{l}\Big\|_{B_{p,q}^{s}(\mathbb{R}^{n},|\cdot
|^{\alpha })}\lesssim \Big(\sum_{l=0}^{\infty }\,2^{lsq}\big\Vert f_{l}%
\big\Vert_{L^{p}(\mathbb{R}^{n},|\cdot |^{\alpha })}^{q}\Big)^{1/q}.
\end{equation*}%
and 
\begin{equation*}
\Big\Vert\sum_{l=0}^{\infty }f_{l}\Big\Vert_{F_{p,q}^{s}(\mathbb{R}%
^{n},|\cdot |^{\alpha })}\lesssim \Big\Vert\Big(\sum_{l=0}^{\infty
}2^{lsq}\left\vert f_{l}\right\vert ^{q}\Big)^{1/q}\Big\Vert_{L^{p}(\mathbb{R%
}^{n},|\cdot |^{\alpha })}.
\end{equation*}
\end{lemma}

The following lemma is particular case of the Plancherel-Polya-Nikolskij
inequality given in \cite{Triebel83}.

\begin{lemma}
\label{Bernstein-Herz-ine1}\textit{Let} $R>0$ and $1\leq p\leq \infty $. 
\textit{Then there exists a positive constant }$c>0$\textit{\ independent of 
}$R$\textit{\ such that for all }$f\in L^{p}$\textit{\ with }$\mathrm{supp}$ 
\textit{\ }$\mathcal{F}f\subset \{\xi \in \mathbb{R}^{n}:|\xi |\leq R\}$%
\textit{, we have} 
\begin{equation*}
\big\|f\big\|_{\infty }\leq c\text{ }R^{\frac{n}{p}}\big\|f\big\|_{p}.
\end{equation*}
\end{lemma}

\begin{proof}[\textbf{\ Proof of Theorem \protect\ref{Triviality1 copy(1)}}]
	
	Let $f\in F_{p,q}^{s}(\mathbb{R}^{n},|x|^{\alpha })$. Recall that%
	\begin{equation*}
	T_{G}(f)=\Pi _{1}(f)+\Pi _{2}(f)+\Pi _{3}(f).
	\end{equation*}%
	Hence we need only to estimate $\Pi _{i}(f),i=1,2$. Lemma \ref{Key-lemma1}
	yields that $\big\Vert\Pi _{1}(f)\big\Vert_{F_{p,q}^{s}(\mathbb{R}%
		^{n},|x|^{\alpha })}$ is bounded by%
	\begin{align*}
		&c\Big\Vert\Big(\sum_{j=2}^{\infty }2^{sjq}\left\vert Q_{j-2}f\Lambda
		_{j}f\right\vert ^{q}\Big)^{1/q}\Big\Vert_{L^{p}(\mathbb{R}%
			^{n},|x|^{\alpha })} \\
		&\lesssim \Big\Vert\sup_{j=2,3,...}|Q_{j-2}f|\Big(\sum_{j=2}^{\infty
		}2^{sjq}\left\vert \Lambda _{j}f\right\vert ^{q}\Big)^{1/q}\Big\Vert%
		_{L^{p}(\mathbb{R}^{n},|x|^{\alpha })} \\
		&\lesssim \sup_{j=2,3,...}\big\Vert Q_{j-2}f\big\Vert_{\infty }\big\Vert f%
		\big\Vert_{F_{p,q}^{s}(\mathbb{R}^{n},|x|^{\alpha })}.
	\end{align*}%
	Let $1\leq p<t<\infty $. From Lemma \ref{Bernstein-Herz-ine1} it follows that%
	\begin{align*}
		\big\Vert Q_{j-2}f\big\Vert_{\infty } &\leq \sum_{i=0}^{j-1}\big\Vert\Lambda
		_{i}f\big\Vert_{\infty }\lesssim \sum_{i=0}^{j-1}2^{i\frac{n}{t}}\big\Vert%
		\Lambda _{i}f\big\Vert_{t} \\
		&\lesssim \big\Vert f\big\Vert_{B_{t,1}^{\frac{n}{t}}(\mathbb{R}^{n})},
	\end{align*}%
	where the implicit constant is independent of $j\geq 2$. Using the embeddings%
	\begin{equation}
	F_{p,\infty }^{s}(\mathbb{R}^{n},|x|^{\alpha })\hookrightarrow B_{p,1}^{%
		\frac{n+\alpha }{p}}(\mathbb{R}^{n},|x|^{\alpha })\hookrightarrow B_{t,1}^{%
		\frac{n}{t}}(\mathbb{R}^{n}),  \label{embeddings1}
	\end{equation}%
	see Theorem \ref{embeddings5}, we get 
	\begin{equation*}
	\big\Vert\Pi _{1}(f)\big\Vert_{F_{p,q}^{s}(\mathbb{R}^{n},|x|^{\alpha
		})}\leq \big\Vert f\big\Vert_{F_{p,q}^{s}(\mathbb{R}^{n},|x|^{\alpha })}^{2}.
	\end{equation*}%
	Again, by Lemmas \ref{Key-lemma1}-\ref{Bernstein-Herz-ine1} we obtain that $%
	\big\Vert\Pi _{2}(f)\big\Vert_{F_{p,q}^{s}(\mathbb{R}^{n},|x|^{\alpha })}$
	is bounded by%
	\begin{align*}
		&c\Big\Vert\Big(\sum_{j=0}^{\infty }2^{sjq}\left\vert \overline{\Lambda }%
		_{j}f\Lambda _{j}f\right\vert ^{q}\Big)^{1/q}\Big\Vert_{L^{p}(\mathbb{%
				R}^{n},|x|^{\alpha })} \\
		&\lesssim \Big\Vert\sup_{j\mathbb{\in N}_{0}}|\overline{\Lambda }_{j}f|\Big(%
		\sum_{j=0}^{\infty }2^{sjq}\left\vert \Lambda _{j}f\right\vert ^{q}\Big)^
			{1/q}\Big\Vert_{L^{p}(\mathbb{R}^{n},|x|^{\alpha })} \\
		&\lesssim \sup_{j\mathbb{\in N}_{0}}\big\Vert\overline{\Lambda }_{j}f%
		\big\Vert_{\infty }\big\Vert f\big\Vert_{F_{p,q}^{s}(\mathbb{R}%
			^{n},|x|^{\alpha })} \\
		&\lesssim \big\Vert f\big\Vert_{F_{p,q}^{s}(\mathbb{R}^{n},|x|^{\alpha
			})}^{2}.
	\end{align*}%
	Let $f\in B_{p,q}^{s}(\mathbb{R}^{n},|x|^{\alpha })$. The proof follows by
	the same arguments above, but we use the embeddings 
	\begin{equation*}
	B_{p,q_{1}}^{s}(\mathbb{R}^{n},|x|^{\alpha })\hookrightarrow B_{p,1}^{\frac{%
			n+\alpha }{p}}(\mathbb{R}^{n},|x|^{\alpha })\hookrightarrow B_{t,1}^{\frac{n%
		}{t}}(\mathbb{R}^{n}),
	\end{equation*}%
	instead of $\mathrm{\eqref{embeddings1}}$, where%
	\begin{equation*}
	q_{1}=\left\{ 
	\begin{array}{ccc}
	1, & \text{if} & s=\frac{n+\alpha }{p} \\ 
	q, & \text{if} & s>\frac{n+\alpha }{p}.%
	\end{array}%
	\right. 
	\end{equation*}%
	The proof is complete.
\end{proof}

\begin{remark}
Let\ $1< p<\infty ,1\leq q\leq \infty \ $and $0\leq \alpha <n(p-1)$. Suppose
that $s>\frac{n+\alpha }{p}$ or 
\begin{equation*}
s=\frac{n+\alpha }{p}\quad \text{and}\quad q=1
\end{equation*}
in the case of Besov spaces $B_{p,q}^{s}(\mathbb{R}^{n},|\cdot |^{\alpha })$
. By similar arguments of Theorem \ref{Triviality1 copy(1)} we obtain that
the spaces $A_{p,q}^{s}(\mathbb{R}^{n},|x|^{\alpha })$ are algebras with
respect to pointwise multiplication.
\end{remark}

\subsection*{Acknowledgements}

We thank the referees for carefully reading the paper and for making several
useful suggestions and comments, which improved the exposition of the paper
substantially. This work is found by the General Direction of Higher
Education and Training under Grant No. C00L03UN280120220004 and by The
General Directorate of Scientific Research and Technological Development,
Algeria.

\end{document}